\documentclass[12pt,a4paper]{amsart}
\usepackage{amssymb}
\usepackage{color,graphicx}

\newtheorem{thm}{Theorem}[section]
\newtheorem{cor}[thm]{Corollary}
\newtheorem{lem}[thm]{Lemma}
\newtheorem{ex}[thm]{Example}

\theoremstyle{definition}

\newtheorem{qu}[thm]{Question}
\theoremstyle{remark}

\begin{document}

\title{The Order  Supergraph of the Power Graph of a Finite Group}

\author{ \bf  A. R. Ashrafi$^\ast$ and A. Hamzeh}

\thanks{$^\ast$Corresponding author (Email: ashrafi@kashanu.ac.ir)}

\address{Department of Pure Mathematics, Faculty of Mathematical
Sciences, University of Kashan, Kashan 87317-53153, I. R. Iran}

\dedicatory{}

\begin{abstract}
The power graph $\mathcal{P}(G)$ is a graph with group elements
as vertex set and two elements are adjacent if one is a power of
the other. The order supergraph $\mathcal{S}(G)$ of the power
graph $\mathcal{P}(G)$ is a graph with vertex set $G$  in which
two elements $x, y \in G$ are joined  if $o(x) | o(y)$ or $o(y) |
o(x)$. The purpose of this paper is to study certain properties
of this new graph together with the relationship between
$\mathcal{P}(G)$ and $\mathcal{S}(G)$.

\vskip 3mm


\vskip 3mm

\end{abstract}

\maketitle

\section{Introduction}
All groups in this paper are finite and we will consider only
simple undirected graphs. Suppose $\Gamma$ and $\Delta$ are
graphs in which $V(\Gamma) \subseteq V(\Delta)$ and $E(\Gamma)
\subseteq E(\Delta)$. Then $\Gamma$ is called a \textit{subgraph}
of $\Delta$ and $\Delta$ a \textit{supergraph} for $\Gamma$.
Suppose $x, y$ are vertices of $\Gamma$. The length of a minimal
path connecting $x$ and $y$ is called \textit{topological
distance} between $x$ and $y$.  The maximum topological distances
between vertices of $\Gamma$ is called its \textit{diameter}. The
topological distance between $x$ and $y$ and diameter of $\Gamma$
are denoted by $d_{\Gamma}(x,y)$ and $diam({\Gamma})$,
respectively. The number of edges incident to a vertex $x$ is
called the {\textit degree} of $x$ denoted by $deg(x)$.

There are several kinds of simple undirected graphs associated
with finite groups which are currently of interest in the field.
For a finite group $G$; there are two graphs associated to the
set of elements of $G$, the \textit{power graph $\mathcal{P}(G)$}
and its \textit{order supergraph} which is denoted by
$\mathcal{S}(G)$. This graph is also recorded in literature as
main supergraph \cite{34}. Two elements $x,y \in G$ are adjacent
in the power graph if and only if one is a power of the other.
They are joined to each other in $\mathcal{S}(G)$ if and only if
$o(x) | o(y)$ or $o(y) | o(x)$. The aim of this paper is studying
the following two questions:
\begin{enumerate}
\item Which graph can occur as  $\mathcal{P}(G)$ or
$\mathcal{S}(G)$?

\item What is the structure of $G$ if $\mathcal{P}(G)$ or $\mathcal{S}(G)$ are
given?
\end{enumerate}

Suppose $\Gamma$ is a simple graph with vertex set $V$ and edge
set $E$ and $R$ is a partition of $V$. Define the quotient graph
$\widehat{\Gamma}$ with vertex set $R$. Two vertices $A$ and $B$
in $R$ are adjacent if and only if there exists a vertex in $A$
and another one in $B$ such that they are adjacent in $\Gamma$.
Bubboloni et al. \cite{01} introduced the notion of order graph
of a finite group $G$, $\mathcal{O}(G)$, as a simple graph with
the following vertex and edge sets:
\begin{eqnarray*}
V(\mathcal{O}(G)) &=& \{ o(x) \mid x \in G\},\\
E(\mathcal{O}(G)) &=& \{rs \mid \{ r, s\} \subseteq V(\mathcal{O}(G)) \ \& \ (r \mid s \ or \ s \mid r)\}.
\end{eqnarray*}
It is easy to see that $\mathcal{O}(G)$ is isomorphic to a quotient graph of $\mathcal{S}(G)$ in which the
set $V(\mathcal{S}(G))$ is partitioned into elements with the same order.

The \textit{directed power graph} of groups and semigroups was
introduced by Kelarev and Quinn in their seminal paper \cite{6}.
They defined the directed power graph of a group $G$ to be a
directed graph with vertex set $G$ and all arcs of the form $uv$
where $v$ is a power of $u$. One of the main result of the
mentioned paper gives a very technical description of the
structure of the power graphs of all finite abelian groups. The
same authors  \cite{7} studied the power graph of the
multiplicative subsemigroup of the ring of $n \times n$ matrices
over a skew-field and a subsemigroup of the monoid of row and
column-monomial $n \times n$ matrices over a group with 0. We
refer the interested readers to consult papers \cite{4,5} for
more information about the power graphs of  semigroups.

The undirected power graph of finite groups was introduced by
Chakrabarty et al. \cite{3}. Chakrabarty et al. \cite{3} proved
that the undirected power graph of  a finite group $G$  is
complete if and only if $G$ is a cyclic $p-$group, for  prime
number $p$. Cameron and Ghosh \cite{1} proved that two abelian
groups with isomorphic power graphs must be isomorphic and
conjectured that two finite groups with isomorphic power graphs
have the same number of elements of each order.   This conjecture
responded  affirmatively by Cameron \cite{2}. Mirzargar et al.
\cite{9}, investigated some combinatorial properties of the power
graph of finite groups and in \cite{11}, some properties of the
power graphs of finite simple groups are considered into account.

Suppose $A$ is a simple graph and $\mathcal{G} = \{ \Gamma_a \}_{a \in A}$ is a set of graphs labeled by vertices of $A$. Following Sabidussi  \cite[p. 396]{12s},  the $A-$join of $\mathcal{G}$ is the  graph  $\Delta$  with the following vertex and edge sets:
 \begin{eqnarray*}
V(\Delta) &=& \{(x,y) \ | \  x  \in V(A) \ \& \ y  \in V(\Gamma_x) \},\\
E(\Delta) &=& \{ (x,y)(x^\prime,y^\prime) \ | \ xx^\prime \in E(A) \
or \  else \  x  =  x^\prime \ \& \ yy^\prime \in E(\Gamma_x)\}.
\end{eqnarray*}
It is  easy to see that this graph can be constructed from $A$
by  replacing  each vertex  $a \in V(A)$ by  the graph $\Gamma_a$
and inserting  either all  or  none  of  the possible edges
between  vertices  of $\Gamma_a$ and $\Gamma_b$ depending on
whether or not  $a$  and  $b$ are joined  by an edge  in $A$. If
$ A $ is an $p-$vertex labeled graph then the $A-$join of
$\Delta_1, \Delta_2, \ldots, \Delta_p$ is denoted by $
A[\Delta_1,\Delta_2, \ldots, \Delta_p]$.

Suppose $G$ is a finite group. The \textit{exponent} of  $G$,
$Exp(G)$, is defined to be the least common multiple of its
element orders. This group is called \textit{full exponent} if
there is an element $x \in G$ such that $o(x) = Exp(G)$. The set
of all prime factors of $|G|$, the set of all  element orders of
$G$ and the number of elements of order $i$ in $G$ are denoted by
$\pi(G)$, $\pi_e(G)$ and $\omega_i(G)$, respectively. The
$\varphi(n)$ denotes the Euler totient function. The notation
\textit{complete  graph} of order $n$  is denoted by $K_n$ and
$K_{m,n}$ denotes the \textit{complete bipartite graph} with
parts of sizes $m$ and $n$, respectively. Our other notations are
standard and can be taken from  \cite{12,13}. We encourage the
interested readers to consult papers \cite{8,10} for more
information on this topic.

\section{Main Results}
The \textit{proper power graph $\mathcal{P^\ast}(G)$}
\cite{01,02} and its \textit{proper order supergraph}
$\mathcal{S^\ast}(G)$ are defined as graphs constructed from
$\mathcal{P}(G)$ and $\mathcal{S}(G)$ by removing identity
element of $G$, respectively. We start this section by comparing
the power graph and its order supergraph. It is easy to see that
every non-identity element is adjacent with identity in
$\mathcal{S}(G)$ which proves that $\mathcal{S}(G)$ is connected
and its diameter is at most two. On the other hand, if $G$ has
even order then identity has odd degree and so $\mathcal{S}(G)$
is not Eulerian.

Suppose $\Gamma$ and $\Delta$ are two graphs, $u$ is a vertex of $\Gamma$ and
$v$ is a vertex in $\Delta$. A splice of $\Gamma$ and $\Delta$ at the vertices
$u$ and $v$ is obtained by identifying the vertices $u$ and $v$
in the union of $\Gamma$ and $\Delta$ which is denoted by $S(\Gamma,\Delta,u,v)$
\cite{33}.

\begin{ex}
Consider the dihedral group $D_{2n} = \langle r,s \ | \ r^n = s^2
= 1, srs = r^{-1}\rangle$. If $n$ is a power of $2$ then
obviously $\mathcal{S}(D_{2n})$ is a complete graph. If $n$ is
odd then $sr^i$, $1 \leq i \leq n$, are all involutions of
$D_{2n}$ which forms an $n-$vertex complete subgraph of
$\mathcal{S}(D_{2n})$. Thus, $\mathcal{S}(D_{2n})$ is a splice of
$\mathcal{S}(Z_{n})$ and complete graph $K_n$ at identity element
of dihedral group. Finally, if $n$ is even but not a power of $2$
then $\mathcal{S}(D_{2n})$   can be constructed from
$\mathcal{S}(Z_{n})$ by adding an $n-$clique $T$ such that each
vertex of $T$ is adjacent  with all elements of $\langle r\rangle$
of even order.
\end{ex}

It is obvious that $G \cong H$ implies that $\mathcal{S}(G) \cong \mathcal{S}(H)$.
If $G$ and $H$ are two non-isomorphic $p-$group with the same order then $\mathcal{S}(G)$
and $\mathcal{S}(H)$ are complete and so they are isomorphic. So, the converse is not generally correct.

\begin{thm}
$\mathcal{P}(G)$ = $\mathcal{S}(G)$ if and only if $G$ is cyclic.
\end{thm}

\begin{proof}
Let $G$ be a cyclic group of order $n$  generated by $x$. Since
$\mathcal{P}(G)$ and $\mathcal{S}(G)$ have the same vertex set
$G$, it is enough to prove that $E(\mathcal{S}(G)) \subseteq
E(\mathcal{P}(G))$. Suppose $a, b$ are adjacent vertices in
$\mathcal{S}(G)$. Then $o(a) | o(b)$ or $o(b) | o(a)$. Set $a =
x^m$ and $b = x^k$. Without loss of generality, we can assume
that $o(a) | o(b)$. This implies that $\frac{n}{(n,m)} |
\frac{n}{(n,k)}$ and so $(n,k) | (n,m)$. Hence $(n,k) | m$ which
shows that $x^m \in \langle x^k\rangle$. Therefore, $a$ and $b$
are adjacent in $E(\mathcal{P}(G))$, as desired.

Conversely, we assume that $E(\mathcal{S}(G)) =
E(\mathcal{P}(G))$. Then for each element $x, y \in G$ with $o(x)
= o(y)$ we have $\langle x \rangle = \langle y \rangle$. Choose
$x \in G$. Hence, there is a unique cyclic subgroup of order
$o(x)$ and so each cyclic subgroup is normal. This implies that
each subgroup is normal and so $G$ is abelian or Hamiltonian
\cite{14}. If $G$ is Hamiltonian, then $G$ can be written as the
direct product of the quaternion group $Q_8$, an elementary
abelian $2-$group and a finite abelian group $A$ of odd order.
Since $Q_8$ has more than one cyclic subgroup, $G$ is abelian.
But, every abelian non-cyclic group has at least two subgroups of
a prime order $p$, where $p | |G|$. Therefore, $G$ is cyclic.
\end{proof}

\begin{thm} \label{thm0}
Let $G$ be a finite group. $\mathcal{S}(G)$  is complete if and only
if $G$ is a  $p-$group.
\end{thm}

\begin{proof}
If $G$ is a finite $p-$group then clearly $\mathcal{S}(G)$ is
complete. Conversely, we assume that $p$ and $q$ are two distinct
prime divisors of $G$. Then there are elements $x$ and $y$ of
orders $p$ and $q$, respectively.  Therefore, $x$ and $y$ are not
adjacent in  $\mathcal{S}(G)$, proving the result.
\end{proof}

A group $G$ is said to be periodic if and only if every element
of $G$ has finite order. It is possible to define $\mathcal{S}(G)$, when $G$ is
periodic and since the identity element is again adjacent to all
elements of $G$, $\mathcal{S}(G)$ will be connected.

Suppose $G$ is a group and $\Gamma$ is a simple graph.  The group
$G$ is called an $EPPO-$group, if all elements of $G$ have prime
power order. It is an $EPO-$group, if all elements have prime
order. An independent set for $\Gamma$ is a set of vertices such
that no two of which are adjacent. The cardinality of an
independent set with maximum size is called the independent
number of $\Gamma$, denoted by $\alpha(\Gamma)$.

\begin{thm}\label{2.5}
$|\pi(G)| \leq \alpha(\mathcal{S}(G)) \leq |\pi_e(G)|-1$. The
right-hand equality is attained if and only if $G$ is $EPO-$group.
\end{thm}

\begin{proof}
Suppose $|G| = p_1^{\alpha_1} \cdots p_r^{\alpha_r}$, where $r
\geq 1$ and $p_i$, $1 \leq i \leq r$, are prime numbers. Thus
$|\pi(G)| = r$. Choose elements $g_i$ of order $p_i$, $1 \leq i
\leq r$ and set $A = \{ g_1, \ldots, g_r\}$. Then $A$ is an
independent subset of $G$ and so $|\pi(G)| \leq
\alpha(\mathcal{S}(G))$. We now assume that $\pi_e(G) = \{ a_1,
\ldots, a_k\}$. Define a graph $\Delta$ such that $V(\Delta) =
\pi_e(G)$ and two vertices $x$ and $y$ are adjacent if and only
if $x | y$ or $y | x$.  Define the induced subgraph $R_i$ of $\mathcal{S}(G)$
with the set of all elements of order $a_i$ as vertex set. Then $R_i$ is a complete graph of
order $\omega_{a_i}(G)$. Then $\mathcal{S}(G) = \Delta[R_1, \ldots, R_k] \cong
\Delta[K_{\omega_{a_1}(G)}, \ldots, K_{\omega_{a_k}(G)}]$. If the
order of an element in $K_{\omega_{a_i}(G)}$ divides the order of
an element in $K_{\omega_{a_j}(G)}$ then all elements of
$K_{\omega_{a_i}(G)}$ will be adjacent to all elements of
$K_{\omega_{a_j}(G)}$. As a consequence  $\alpha(\mathcal{S}(G))
\leq |\pi_e(G)| - 1$.

If $G$ is an $EPO-$group then all elements of order $p$, $p$ is
prime, will be a clique. So, $ \alpha(\mathcal{S}(G)) =
|\pi_e(G)| - 1$. Conversely, we assume that $ \alpha(\mathcal{S}(G)) =
|\pi_e(G)| - 1$. If $G$ has a non-identity element of a non-prime order $a$
then there exists an element $b \in \langle a \rangle$ of prime order. So,
each element of order $o(a)$ will be adjacent to each element of order $o(b)$
which contradicts by our assumption. Therefore, all non-identity
elements of $G$ have prime order, as desired.
\end{proof}

It can be easily seen that if $G$ is an $EPPO-$group then $|\pi(G)|
= \alpha(\mathcal{S}(G))$. On the other hand, $2 = |\pi(Z_6)| =
\alpha(\mathcal{S}(Z_6))$, but $Z_6$ is not an $EPPO-$group. So,
the following question remains open:

\begin{qu}
What is the structure of groups with $|\pi(G)| =
\alpha(\mathcal{S}(G))$?
\end{qu}

For a given group $G$, the number of edges in the order supergraph
$\mathcal{S}(G)$ is denoted by $e(\mathcal{S}(G))$. In the
following theorem an exact expression for $e(\mathcal{S}(G)) =
|E(\mathcal{S}(G))|$ is calculated.

\begin{thm}\label{th2.7}
$e(\mathcal{S}(G)) = \frac{1}{2}\sum_{x \in
G}\left(2\sum_{d|o(x)}\omega_d(G)-\omega_{o(x)}(G)-1\right)$.
\end{thm}

\begin{proof}
Define the directed graph $\overrightarrow{\mathcal{S}}(G)$ with
vertex set $G$ and arc set $E(\overrightarrow{\mathcal{S}}(G)) =
\{ (x,g) \ | \ o(g) | o(x)\}.$ Suppose $x \in G$. Then $$Outdeg
(x)=\left|\{ g\in G,o(g)|o(x)\} \right|=\sum _{d|o(x)}\omega _{d}
(G)-1.$$ It is clear that the whole number of arcs is equal to
the sum of $Outdeg(x)$ overall vertices of
$\overrightarrow{\mathcal{S}}(G)$.

On the other hand, $\sum_{x \in
V(\overrightarrow{\mathcal{S}}(G))}Outdeg(x) =
2e(\overrightarrow{\mathcal{S}}(G))$. To compute the number of
edges in $\mathcal{S}(G)$, it is enough to count once the edges of
$\overrightarrow{\mathcal{S}}(G)$ with two different directions.
But we have an undirected edge connecting $g$ and $x$, when $o(g)
| o(x)$ and $o(x) | o(g)$. In such a case, we will have $o(x) =
o(g)$. Therefore, the number of edges with two different
directions is $\omega_{o(x)}(G)-1$ and so $$e(\mathcal{S}(G)) =
\frac{1}{2}\sum_{x \in
G}\left(2\sum_{d|o(x)}\omega_d(G)-\omega_{o(x)}(G)-1\right),$$
proving the result.
\end{proof}

In the following corollary, we apply the previous theorem to present a new proof for \cite[Corollary 4.3]{3}.

\begin{cor}
$2e(\mathcal{P}(Z_n)) = \sum_{d|n}\left[ 2d - \varphi(d) - 1\right]\varphi(d)$.
\end{cor}

\begin{proof}
It is easy to check the following formulas:
\begin{eqnarray*}
\sum_{x \in Z_n}\sum_{d | o(x)} \omega_d(Z_n) &=& \sum_{d | n}d\varphi(n),\\
\sum_{x \in Z_n}\left[\omega_{o(x)}(Z_n) - 1\right] &=& \sum_{d | n}\varphi^2(d) - \sum_{d | n}\varphi(d).
\end{eqnarray*}
Therefore, by Theorem \ref{th2.7} and the fact that $\mathcal{P}(Z_n) = \mathcal{S}(Z_n)$, $2e(\mathcal{P}(Z_n))$ = $2e(\mathcal{S}(Z_n))$ = $\sum_{d|n}\left[ 2d - \varphi(d) - 1\right]\varphi(d)$, as desired.
\end{proof}

\begin{thm}
Let $G$ be a finite group of order $p_{1}^{n_{1}}\cdots
p_{k}^{n_{k}}$ and $$V_{i} =\{ g\in G \ | \ g \ne 1,
|g|\left|p_{i} ^{n_{i} } \right.\}.$$ $G$ is an $EPPO-$group if
and only if $\mathcal{S}(G) = K_{1}+(\bigcup_{i=1}^{k}K_{|V_{i}|}
)$.
\end{thm}

\begin{proof}
Suppose $G$ is an $EPPO-$group then by the structure of the order
supergraph of $G$, each $V_i$ is a clique in  $\mathcal{S}(G)$
and there is no edge connecting $V_i$ and $V_j$, $i \ne j$. Hence
$\mathcal{S}(G) = K_{1}+(\bigcup_{i=1}^{k}K_{|V_{i}|} )$.

Conversely, we assume that $\mathcal{S}(G) =
K_{1}+(\bigcup_{i=1}^{k}K_{|V_{i}|} )$, $e \ne g \in G$ and
$p_ip_j | |G|$. We also assume that $g_i$ and $g_j$ are two
elements in $G$ such that $o(g_i) = p_i$ and $o(g_j) = p_j$.
Obviously, $g_i \in V_i$ and $g_j \in V_j$ and so $g_i \in
K_{|V_i|}$ and $g_j\in K_{|V_j|}$. Since $|g_i| | |g|$ and $|g_j|
| |g|$, $g$ is adjacent to $g_i$ and $g_j$, which is impossible.
Therefore, $G$ is an $EPPO-$group, as desired.
\end{proof}

\begin{cor}\label{co9}
Suppose $G$ is an $EPPO-$group which is not a $p-$group. Then the
number of components in $\mathcal{S}^\ast(G)$ is equal to
$|\pi(G)|$.
\end{cor}

Following Williams \cite{16}, we assume that $G$ is a finite
group and construct its prime graph as follows: the vertices are
the primes dividing the order of the group, two vertices $p$ and
$q$ are joined by an edge if and only if $G$ contains an element
of order $pq$.

\begin{cor}\label{2.10}
If the prime graph of a group $G$ is totally disconnected then
$\mathcal{S}^\ast(G)$ is disconnected.
\end{cor}

\begin{proof}
It is clear that $G$ is an $EPPO-$group if and only if its prime
graph is totally disconnected. We now apply the Corollary \ref{co9} to
deduce the result.
\end{proof}

The vertex connectivity of a graph $\Gamma$ is the minimum number of
vertices, $\kappa(\Gamma)$, whose deletion from $\Gamma$ disconnects it.

\begin{thm} \label{thm1}
The vertex connectivity of $\mathcal{S}(G)$ can be computed as
follows:
\begin{itemize}
\item $\kappa(\mathcal{S}(Z_n)) = n-1$, where $n = p^m$, $p$ is
prime and $m$ is a non-negative integer.

\item Suppose $n$ is not prime power. Then  $\kappa(\mathcal{S}(Z_n)) \geq \varphi(n) +
1$.
 The equality is satisfied if and only if $n = pq$, $p$
and $q$ are distinct prime numbers.
\end{itemize}
\end{thm}

\begin{proof}

If $n = p^m$ then  the graph $\mathcal{S}(G)$ is an $n-$vertex
complete graph and so $\kappa(\mathcal{S}(Z_n)) = n-1$. To prove
the second part, we have to note that if an $n-$vertex graph
$\Gamma$ has exactly $s$ elements of degree $n-1$ then $\kappa(G)
\geq s$. In fact, we have to delete all such elements to find
a disconnected graph. Since  $n$ is not prime power and the
cyclic group $Z_n$ has exactly $\varphi(n)$ generators, it has at
least $\varphi(n) + 1$ elements of degree $n-1$. Therefore,
$\kappa(\mathcal{S}(Z_n)) \geq \varphi(n) + 1$. If the equality
is satisfied and $n$ is divisible by at least three different
primes then by deleting the identity and all
elements of order $n$ the resulting graph $H$ will be connected. To
prove, we first assume that $n$ is not a square free integer with exactly three prime factors. Choose two elements $g$ and $h$ in $H$ such that $o(g) = p_1^{\alpha_1}p_2^{\alpha_2}\cdots p_r^{\alpha_r}$,  $r \geq 3$, and   $o(h) =
q_1^{\beta_1}q_2^{\beta_2}\cdots q_t^{\beta_t}$, where $p_i$'s
and  $q_j$'s are prime numbers;  $\alpha_i$'s and $\beta_j$'s are
positive integers. If $u$ is an element of order $p_1q_1$ then
$(o(g),o(u)) = p_1$ and $(o(h),o(u)) = q_1$, there are paths
connecting $u$ and $g$ as well as $u$ and $h$. This implies that
there exists a path connecting $g$ and $h$ and so
$\kappa(\mathcal{S}(Z_n)) > \varphi(n) + 1$. So, it is enough to check the case that $n = pqr$. Since $Z_n$ has elements of orders $pq$, $pr$ and $qr$, all elements of the graph obtained from $\mathcal{S}(Z_n)$ by deleting the identity and all elements of order $pqr$ will be again connected and so
$\kappa(\mathcal{S}(Z_n)) > \varphi(n) + 1$.  Thus by our assumption, $n$  has exactly two prime factors, say $n = p^kq^l$,
where $p, q$ are distinct primes and $k, l$ are positive
integers. If $k \geq 2$ or $l \geq 2$ then by choosing an element
of order $pq$ and applying a similar argument as above, we can
see again the resulting graph will be connected. Therefore, $n =
pq$, where $p$ and $q$ are distinct prime numbers. Conversely, if $n=pq$ then
clearly $\kappa(\mathcal{S}(Z_n)) = \varphi(n) + 1$, which
completes the proof.
\end{proof}

The Kuratowski's theorem states that  a finite graph is planar if
and only if it does not contain a subgraph that is a subdivision
of the complete graph $K_5$ or of the complete bipartite graph
$K_{3,3}$. Here,  a subdivision of a graph $\Gamma$ is a graph
resulting from the subdivision of edges in $\Gamma$. In what follows,
we apply this theorem to give a classification of the planar order
supergraph of a finite group.

\begin{thm}\label{2.12}
The order supergraph of a finite group $G$ is planar  if and only
if $G \cong 1, Z_2, Z_3, Z_4, Z_2 \times Z_2$ or $S_3$.
\end{thm}

\begin{proof}
Suppose the order supergraph of  $G$ is planar,  $n = Max\{ o(x) \
| \ x \in G\}$ and $x \in G$ is an element of order $n$. If  $5 |
n$ or $n$ is a prime number $\geq 5$ then  $\mathcal{S}(G)$ has a
subgraph isomorphic to $K_5$. On the other hand, $\langle
x\rangle$ has exactly $\varphi(n)$ generators and so
$\mathcal{S}(G)$ has a subgraph isomorphic to $K_{\varphi(n) +
1}$. Apply Kuratowski's theorem, we deduce that $n =
2^\alpha3^\beta$. We claim that $\alpha \leq 2$ and $\beta \leq
1$. Otherwise, $\mathcal{S}(G)$ has an induced subgraph
isomorphic to $K_5$, a contradiction. Thus,  $n \in \{ 1, 2, 3,
4,  6, 12\}$. If $G$ has an element of order $6$ or $12$ then we
will have a again a subgraph isomorphic to $K_5$ and so $\pi_e(G)
\subseteq \{ 1, 2, 3, 4\}$. Since elements of the same order in
$G$ constitute a clique in $\mathcal{S}(G)$, $\omega_i(G) < 4$,
when $i \in \{ 2, 3, 4\}$. Again since all elements of order $2$
are adjacent to all elements of order $4$, $\omega_2(G) +
\omega_4(G) < 4$. Since $\pi_e(G) \subseteq \{ 1, 2, 3, 4\}$ and
$\omega_i(G) < 4$, $2 \leq i \leq 4$, by counting the number of
elements of each order we have $|G| \leq$ $1 + 3 + 3 + 3 = 10$ and
a simple calculation by small group library of  GAP \cite{15}
proves that $G$ is isomorphic to $1, Z_2, Z_3, Z_4, Z_2 \times
Z_2$ or $S_3$. The converse is clear.
\end{proof}

By previous Theorem , the graph $\mathcal{S}(Z_n)$ is planar if
and only if $n < 5$. Suppose $\Gamma$ is a finite graph. The
clique number of $\Gamma$, $\omega(\Gamma)$, is the size of a
maximal clique in $\Gamma$ and the chromatic number of $\Gamma$,
$\chi(\Gamma)$, is the smallest number of colors needed to color
the vertices of $\Gamma$ so that no two adjacent vertices share
the same color. It is clear that $\chi(\Gamma)$ $\geq$
$\omega(\Gamma)$.

\begin{cor}
If the order supergraph of a finite group $G$ is planar then
$\chi(\mathcal{S}(G)) = \omega(\mathcal{S}(G))$.
\end{cor}

\begin{thm}
The proper order supergraph of a finite group $G$ is planar if and
only if $G \cong 1, Z_2, Z_3, Z_4, Z_2 \times Z_2, Z_5, Z_6$ or
$S_3$.
\end{thm}

\begin{proof}
Suppose $\mathcal{S}^\star(G)$ is planar, $n = Max\{ o(g) \ | \ g
\in G\}$ and $x \in G$ has order $n$. Since $\langle x \rangle$
has exactly $\varphi(n)$ elements of order $n$ and these elements
constitute a clique in $\mathcal{S}^\star(G)$, $\varphi(n) < 5$.
On the other hand, if $n$ is a prime power and $n \geq 7$ then
$\mathcal{S}^\star(G)$ has a clique isomorphic to $K_5$ which is
impossible. Thus $\pi(G) \subseteq \{ 2, 3, 5\}$ and $\pi_e(G)
\subseteq \{ 1, 2, 3, 4, 5, 6, 10, 12\}$. If $G$ has an element
$x$ of order $10$ or $12$ then the set $A$ containing all
generators of $\langle x \rangle$ and $x^2$ will be a clique of
order $5$ which leads us to another contradiction. This shows
that $\pi_e(G) \subseteq \{ 1, 2, 3, 4, 5, 6\}$. Since elements
of the same order is a clique in the planar graph
$\mathcal{S}^\star(G)$, $\omega_i(G) \leq 4$, for $i \leq 6$.
Finally, the elements of orders $2, 4$, the elements of orders $2,
6$ and the elements of orders $3, 6$ constitute three cliques in
$\mathcal{S}^\star(G)$ and so $\omega_2(G) + \omega_4(G) \leq 4$,
$\omega_2(G) + \omega_6(G) \leq 4$ and $\omega_3(G) + \omega_6(G)
\leq 4$. Since $\pi_e(G) \subseteq \{ 1, 2, 3, 4, 5, 6\}$ and $\omega_i(G) \leq 4$, $2 \leq i \leq 6$, by counting the number of elements of each order $|G| \leq$ 1 + 4 + 4 + 4 + 4 + 4 = 21 and a simple calculation by small
group library of  GAP proves that $G \cong 1, Z_2, Z_3, Z_4, Z_2
\times Z_2, Z_5, Z_6$ or $S_3$. The converse is obvious.
\end{proof}

Suppose $G$ is a finite group and $S(G)$ is the set of   all
elements $x \in G$ such that $x$ is adjacent to all elements of
$G \setminus \{ x\}$ in $\mathcal{S}(G)$. In the following
theorem, the groups $G$ with $|S(G)| > 1$ is characterized.

\begin{thm}
$|S(G)| > 1$ if and only if $G$ is a non-trivial  full exponent finite group.
\end{thm}

\begin{proof}
Suppose  $|S(G)| > 1$ and $e \ne g \in S(G)$. If $G$  is a
$p-$group then $G$ is full exponent, as desired. We assume that
$G$ does not have prime power order and $p$ is a prime factor of
$|G|$. Then obviously $p | o(g)$ and so $o(g)$ is divisible by
all prime factors of $|G|$. Choose an element $x \in G$ of order
$p^r$, where $p$ is a prime number and $r$ is a non-negative
integer. Since $deg(g) = |G| - 1$, $g$ and $x$ are adjacent. So,
$o(x) | o(g)$ or $o(g) | o(x)$. By our assumption, $G$ does not
have prime power order and hence $o(x) | o(g)$. This shows that
$o(g) = Exp(G)$ and therefore $G$ is full exponent.

Conversely, we assume that $G$ is full exponent and $a$ is  an
element of $G$ such that $o(a) = Exp(G)$. Thus $a$ is adjacent to
all elements of $G \setminus \{ a\}$. This proves that $|S(G)| >
1$, proving the result.
\end{proof}

\begin{cor}
If $G$ is isomorphic a nilpotent group, a dihedral group  of
order $2n$ with even $n$, $H \times Z_n$ such that $H$ is
arbitrary finite group and $n = Exp(H)$ or $H^n$ with $n =
|\pi(H)|$ then $|S(G)| > 1$.
\end{cor}

\begin{thm}
The order supergraph of a finite group is bipartite if and only if $G \cong 1, Z_2$.
\end{thm}

\begin{proof}
If $|G|$ has an odd prime factor $p$ and $x \in G$ has order
$p$, then $\langle x \rangle$ is a subgraph isomorphic to $K_p$
which is impossible. Thus $G$ is a $2-$group and by Theorem
\ref{thm0}, $G$ is complete. This shows that $n \leq 2$, as
desired.
\end{proof}

\begin{cor}
The order supergraph of a finite group is a tree if and only if $G \cong 1, Z_2$.
\end{cor}

\begin{thm}
Let $G$ be a non-trivial finite group. Then  $\mathcal{S}^\ast(G)$ is bipartite if and only if $G \cong  Z_2$ or $Z_3$.
\end{thm}

\begin{proof}
Suppose $\mathcal{S}^\ast(G)$ is bipartite and $G$ has an
element $x$ of order $\geq 4$. It is well-know that if $n \geq 4$ then $\varphi(n) \geq 2$. Hence the cyclic group $\langle x
\rangle$ has at least two generators $x$, $x^t$, where $t > 1$. Choose an element $x^j$ different from $x$ and $x^t$ and so $j \ne 1, t$. Since $o(x^j) | o(x) = o(x^t)$, $x, x^j, x^t, x$ is a cycle of length $3$, contradicts
by bipartivity of $\mathcal{S}^\ast(G)$. Thus $\pi_e(G) \subseteq
\{ 1, 2, 3\}$. On the other hand, since all element of a given
order constitute a clique in $\mathcal{S}^\ast(G)$, $\omega_i(G) \leq 2$, $i = 2, 3$. This implies that $|G| \leq 5$.
Therefore, $G \cong Z_2$ or $Z_3$.
\end{proof}

\begin{cor}
The graph  $\mathcal{S}^\ast(G)$ is tree if and only if $G \cong  Z_2$ or $Z_3$.
\end{cor}

Suppose $\Gamma$ is a connected graph with exactly $n$ vertices and $m$ edges. If $c = m - n + 1$ then $\Gamma$
is called $c-$cyclic. It is easy to see that $c=0$ if and only if $\Gamma$ is a tree. In the case that
$c = 1, 2, 3, 4$ and $5$, we call $\Gamma$ to be unicyclic, bicyclic, tricyclic, tetracyclic and pentacyclic, respectively.

\begin{lem}\label{thm35}
The order supergraph $\mathcal{S}(G)$ has a pendant vertex if and only if $|G| \leq 2$.
\end{lem}

\begin{proof}
Suppose $|G| > 1$ and $x$ is a non-identity element such that $deg_{\mathcal{S}(G)}(x) = 1$. If $o(x) > 2$
then there exists $y \in \langle x \rangle$ such that $xy, xe \in E(\mathcal{S}(G))$ which is impossible.
Hence $G$ is elementary abelian and in such a case all involutions are adjacent. This shows that $G \cong Z_2$,
proving the lemma.
\end{proof}

\begin{thm} Let $G$ be a finite group. The following are hold:
\begin{enumerate}
\item The order supergraph $\mathcal{S}(G)$ is unicyclic if and only if $G \cong Z_3$.

\item The order supergraph $\mathcal{S}(G)$ cannot be bicyclic and pentacyclic.

\item The order supergraph $\mathcal{S}(G)$ is tricyclic if and only if $G \cong Z_4$ or $Z_2 \times Z_2$.

\item The order supergraph $\mathcal{S}(G)$ is tetracyclic if and only if $G \cong D_6$.
\end{enumerate}
\end{thm}

\begin{proof} Suppose $|G| = n$. Our main proof will have some separate cases as follows:
\begin{enumerate}
\item Assume that $\mathcal{S}(G)$ is unicyclic. Thus $m = n$. Since $deg(e) = n - 1$ the edges in $\mathcal{S}(G)$ are those adjacent to $e$
and only a further edge. Thus $\mathcal{S}(G)$ is a triangle or it contains some pendant vertex. In that last case, by Lemma \ref{thm35}
we have $G \cong Z_2$ and the graph is a tree, a contradiction. Thus $\mathcal{S}(G)$ is necessarily a triangle. This gives $|G| = 3$ and
thus $G \cong Z_3$. The converse is clear.

\item Since $\mathcal{S}(G)$ is bicyclic, $\mathcal{S}^\ast(G)$ has exactly two edges.  By Lemma \ref{thm35},
if these edges are  incident then $|G| = 4$ and in another case we have  $|G| = 5$. In both case, the resulting order supergraph will not be bicyclic, a contradiction. Assume next that $\mathcal{S}(G)$ is pentacyclic. Then $G\ncong 1, Z_2$ and thus, by Lemma \ref{thm35}, $\mathcal{S}(G)$ has no pendant
vertex. The graph $\mathcal{S}^\ast(G)$ contains five edges. Those edges can be incident or not. If they are not incident two by two
then $\mathcal{S}(G)$ is a friendship graph on 11 vertices. On the other hand, $\mathcal{S}^\ast(G)$ must contain at least four vertices because if there are only three vertices there are at most three edges. It follows that $5 \leq |G| \leq 11$. A case by case investigation of all groups with orders  $5 \leq |G| \leq 11$ show that there is no group $G$ with pentacyclic order supergraph.

\item Since $\mathcal{S}(G)$ is tricyclic, $\mathcal{S}^\ast(G)$ has exactly three edges. A case by case investigation shows that we have five possible graph structure for $\mathcal{S}(G)$ depending on that those edges can be incident or not. On the other hand, a similar argument as case (2) shows that $\mathcal{S}(G)$ must contain at least four and  at most seven vertices, i.e. $4 \leq |G| \leq 7$. By considering all groups of these orders, one can easily prove that $|G| = 4$, as desired.

\item Since $\mathcal{S}(G)$ is tetracyclic, $\mathcal{S}^\ast(G)$ has exactly four edges. A similar argument like other cases shows that $5 \leq |G| \leq 9$ and a case by case investigation lead us to the result.
\end{enumerate}
This completes the proof.
\end{proof}

\begin{thm}
Let $G$ be a finite group and $e = xy \in E(\mathcal{S}^\ast(G))$. $e$ is a cut edge  if and only if $\{ x, y\}$
is a component of $\mathcal{S}(G)$.
\end{thm}

\begin{proof}
Suppose $e = xy$ is a cut edge for $\mathcal{S}^\ast(G)$. Without loss of generality, we
assume that $o(y) | o(x)$. Suppose $o(x) = 2$. Then $o(y) = 2$ and obviously $G$ does not have another element of
even order. This shows that $\{ x, y\}$ is a component of $\mathcal{S}^\ast(G)$. If $o(x) =3$ then $o(y) = 3$ and the
same argument shows that $\{ x, y\}$ is a component of $\mathcal{S}^\ast(G)$. So, we can assume $o(x) \geq 4$.
In this case, since  $\varphi(o(x)) \geq 2$, the cyclic subgroup $\langle x \rangle$ has at least two generators
and once more element $z$ such that $o(z) | o(x)$. This implies that $e$ is an edge of a cycle which is impossible.
\end{proof}

The graph $\Gamma$ is called perfect, if $\chi(\Gamma) = \omega(\Gamma)$, where $\chi(\Gamma)$ and $\omega(\Gamma)$ denote the chromatic
and clique numbers of $G$, respectively.

\begin{thm} \label{cho} {\rm (}Chudnovsky et al. \cite{32}{\rm )}
The graph $\Gamma$ is perfect if and only if $\Gamma$ and $\overline{\Gamma}$ don't have an induced odd cycle of length at least five.
\end{thm}

\begin{thm}
The order supergraph of a finite group is perfect.
\end{thm}

\begin{proof}
Suppose $\mathcal{S}(G)$ has an induced odd cycle $C$ of length
at least five and $\vec{C}$ is its associated directed cycle in
the directed order supergraph of $G$. Choose a directed path $x
\rightarrow y \rightarrow z$ in $\vec{C}$. Then $o(y) | o(x)$ and
$o(z) | o(y)$ which implies that $xz \in E(\mathcal{S}(G))$.
Therefore, $\mathcal{S}(G)$ has an  odd cycle of length $3$
contradicts by Theorem \ref{cho}.

We now assume that $\overline{C}$ is an induced odd cycle of length at least five in
$\overline{\mathcal{S}(G)}$. If $l(\overline{C}) = 5$ then $C$
is an induced cycle of length $5$, contradicts by the first part of proof. So, we can assume
that $l(\overline{C}) \geq 7$. Then $C$ contains a triangle. On the other hand, for
each $x \in V(\vec{C})$, $Indeg(x) = 0$ or $Outdeg(x) = 0$ which is not correct for a triangle.
This leads us to another contradiction. Hence by Theorem \ref{cho} the order supergraph $\mathcal{S}(G)$ is perfect.
\end{proof}

In what follows, we apply results of \cite{01,02} to compute  the
number of connected components of $\mathcal{S}^\ast(A_n)$ and
$\mathcal{S}^\ast(S_n)$, where $A_n$ and $S_n$ denote as usual
the alternating and symmetric groups on $n$ symbols. In our next
results, $P$ stands for the set of all prime numbers and $aP + b
= \{ ap + b \ | \ p \in P\}$.

\begin{lem}\label{lem25}
If $n \in \{ p, p+1\}$,  $p \in P$, then elements of order $p$ is a component of
$\mathcal{S}^\ast(S_n)$. If $n \geq 4$ and $n \in \{ p, p+1, p+2\}$,  $p \in P$, then elements of
order $p$ is a component of $\mathcal{S}^\ast(A_n)$.
\end{lem}

\begin{proof}
The proof follows from the fact that the groups $A_n$ and $S_n$ do not have elements of order $kp$, $k \geq 2$.
\end{proof}

\begin{lem}\label{lem26} {\rm (Bubboloni et al. \cite[Page 24]{01})}
Suppose $B_1(n) = P \cap \{n, n-1\}$ and $\Sigma_n = \{ g\in S_n \ | \ o(g) \not\in B_1(n)\}$, $n \geq 8$.
Then $\Sigma_n$ is a component of $\mathcal{S}^\ast(S_n)$.
\end{lem}

Set $A = P \cup (P+1) \cup (P+2) \cup 2P \cup (2P + 1)$. For the sake of completeness, we mention here
a result of \cite{02}.

\begin{thm} \label{thm27} {\rm (Bubboloni et al. \cite[Corollary C]{02})}
The proper power graph $\mathcal{P}^\ast(A_n)$ is connected if and only if $n = 3$ or $n \not\in A $.
\end{thm}

\begin{cor} \label{2.29}
If $n = 3$ or $n \not\in A$ then $\mathcal{S}^\ast(A_n)$ is connected.
\end{cor}

\begin{proof}
If $n = 3$ or $n \not\in A$ then by Theorem \ref{thm27}, the proper
power graph $\mathcal{P}^\ast(A_n)$ is connected. The result now
follows from this fact that  $\mathcal{P}^\ast(A_n)$ is a
subgraph of $\mathcal{S}^\ast(A_n)$.
\end{proof}

Define $B_2(n) = \{ n \mid n \in P, n-1 \in P, n-2 \in P, \frac{n}{2} \in P or \frac{n-1}{2} \in P\}$, $\theta_n = \left(\pi_e(A_n) \setminus \{
1\}\right) \setminus B_2(n)$ and  $\overline{\xi}_n = \{ g \in
A_n \ | \ o(g) \in \theta_n\}$. In what follows, the number of
connected components of a graph $\Gamma$ is denoted by
$c(\Gamma)$. The elements of $B_2(n)$ are called the critical
orders of $n$.

\begin{cor}\label{lem30}{\rm (Bubboloni et al. \cite[Page 15]{02})}
There exists a component in $\mathcal{S}^\ast(A_n)$ containing
$\overline{\xi}_n$.
\end{cor}

\begin{thm}
The proper order supergraph $\mathcal{S}^\ast(A_n)$ is connected
if and only if $n = 3$ or  $n, n-1$ and $n-2$ are not prime integers. The
maximum number of connected component is $3$ and the graph
$\mathcal{S}^\ast(A_n)$ attained the maximum possible of
component if and only if $n \in P \cap (P + 2)$.
\end{thm}

\begin{proof}
 We first assume that $n \leq 10$. The graph
$\mathcal{S}^\ast(A_2)$ is empty  and obviously
$c(\mathcal{S}^\ast(A_3)) = 1$. Our calculations by GAP \cite{15}
show that
\begin{eqnarray*}
\pi_e(A_4) &=& \{ 1, 2, 3\}, \\
\pi_e(A_5) &=& \{ 1, 2, 3, 5\},\\
\pi_e(A_6) &=& \{ 1, 2, 3, 4, 5\},\\
\pi_e(A_7) &=& \{ 1, 2, 3, 4, 5, 6, 7\},\\
\pi_e(A_8) &=& \{ 1, 2, 3, 4, 5, 6, 7, 15\},\\
\pi_e(A_9) &=& \{ 1, 2, 3, 4, 5, 6, 7, 9, 10, 12, 15\},\\
\pi_e(A_{10}) &=& \{ 1, 2, 3, 4, 5, 6, 7, 8, 9, 10, 12, 15, 21\}.
\end{eqnarray*}
From the set of element orders one can see that
$c(\mathcal{S}^\ast(A_4)) = 2$,  $c(\mathcal{S}^\ast(A_5)) =
c(\mathcal{S}^\ast(A_6))$ = $c(\mathcal{S}^\ast(A_7)) = 3$,
$c(\mathcal{S}^\ast(A_8))$ = $c(\mathcal{S}^\ast(A_9)) = 2$ and
$c(\mathcal{S}^\ast(A_{10})) = 1$.
If $n \geq 11$ then a similar argument as the proof of \cite[Theorem A and Corollary C]{02} will complete the proof.
\end{proof}

\begin{thm} The number of connected components of  $\mathcal{S}^\ast(S_n)$ can be computed as follows:
\begin{enumerate}
\item $c(\mathcal{S}^\ast(S_2)) = 1$ and $c(\mathcal{S}^\ast(S_3)) = c(\mathcal{S}^\ast(S_4)) =
c(\mathcal{S}^\ast(S_5)) =  c(\mathcal{S}^\ast(S_6)) = c(\mathcal{S}^\ast(S_7)) = 2$.

\item If $n \geq 8$
then  $c(\mathcal{S}^\ast(S_n)) = \left\{\begin{array}{ll} 1 & n \not\in P \cup (P+1)\\ 2 & otherwise\end{array}\right.$.
\end{enumerate}
\end{thm}

\begin{proof}
The proof is completely similar to the proof of \cite[Theorem B and Corollary C]{01} and so it is omitted. 
\end{proof}

\begin{cor} (See Bubboloni et al. \cite[Corollary C]{01})
If $n \geq 2$ then $\mathcal{S}^\ast(S_n)$ is $2-$connected if and only if $n = 2$ or $n \not\in P \cup (P + 1)$.
\end{cor}

Since the power graph of a group $G$ is a spanning subgraph of
the order supergraph of $G$, it is easy to see that
$2-$connectedness of the power graph implies the same property for
the order supergraph. Akbari and Ashrafi \cite{0}, conjectured
that the power graph of a non-abelian finite simple group $G$ is
$2-$connected if and only if $G \cong A_n$, where $n=3$ or
$n\not\in  P \cup (P+1) \cup (P+2) \cup 2P \cup (2P + 1)$.
 They verified this conjecture in some classes of simple groups. We now apply
the order supergraph to prove that the power graph of some other classes of simple groups are not $2-$connected.

\begin{thm}
The order supergraph of sporadic groups are not $2-$connected.
\end{thm}

\begin{proof}
It is easy to see that if we can partition $\pi_e(G)$ into
subsets $L(G)$ and $K(G)$ such that any element of $L(G)$ is
co-prime with any element of $K(G)$ then order supergraph of $G$ is
not $2-$connected. In Table 1, the set $L(G)$ for a sporadic group
$G$ is computed. By this table, one can see that there is no
sporadic simple group with $2-$connected order supergraph.
\end{proof}

\vskip 3mm

\begin{center}
\begin{tabular}{l|l|l}\hline
$L(M_{11} )=\{ 11,5\}$ & $L(M_{12} ) = \{ 11\}$ & $L(McL)=\{
11\}$\\  $L(M_{23}) = \{ 11,23\}$ & $L(M_{24} )= \{ 11,23\}$ &
$L(Co_{2}) = \{ 11,23\}$\\ $L(HS) = \{ 7,11\}$ &  $L(J_{2}) = \{
7\}$ & $L(Co_{1} )=\{ 23\}$\\ $L(Co_{2})=\{ 23\}$ & $L(Suz)=\{
11,13\}$ & $L(He)=\{ 17\}$\\
$L(HN)=\{ 19\}$ & $L(Th)=\{ 13,19,31\}$ & $L(Fi_{22} )=\{ 13\}$\\
$L(Fi_{23} )=\{ 17,23\}$ & $L(Fi^\prime_{24} ) = \{ 17,23,29\}$ &
$L(B) = \{ 31,47\}$\\ $L(M)=\{ 41,71,59\} $ & $L(J_{1} )=\{
7,11,19\} $ & $L(O^\prime N)=\{ 11,19,31\}$\\ $L(J_{3} )=\{
17,19\}$ & $L(Ru)=\{ 29\} $ & $L(J_{4} )=\{ 23,31,29,37,43\}$\\
$L(M_{22} )=\{ 5,7,11\} $ & $L(Ly)=\{ 31,37,67\} $ & \\ \hline
\end{tabular}
\vskip 3mm \centerline{{\bf Table 1.} The Set $L(G)$, for all
Sporadic Groups $G$.}
\end{center}

If two elements $x$ and $y$ in a finite group $G$ have co-prime order then obviously they
cannot be adjacent in $\mathcal{P}(G)$. In \cite{0}, the authors applied this simple
observation to prove that the power graph of some classes of simple groups are not connected.
The following result is an immediate consequence of \cite[Theorems 2-10]{0}.

\begin{thm}
The order supergraph of the following simple groups are not $2-$connected:

\begin{enumerate}
\item $^2F_4(q)$, where $q = 2^{2m+1}$ and $m \geq 1$;

\item $^2G_2(q)$, where $q = 3^{2m+1}$ and $m \geq 0$;

\item $A_1(q)$, $A_2(q)$, $B_2(q)$, $C_2(q)$ and $S_4(q)$, where $q$ is an odd prime power;

\item $F_4(2^m)$, $m \geq 1$ and $U_3(q)$, where $q$ is a prime power.
\end{enumerate}
\end{thm}

In the end of this paper, we have to say that it remains an open question to classify all simple groups with $2-$connected order supergraph.  

\vskip 3mm

\noindent{\bf Acknowledgement.} The authors are indebted to the referee for his/her comments leaded us to correct some results and improve our arguments. This research is supported by INSF under grant number 93045458.


\vskip 3mm

\begin{thebibliography}{33}
\bibitem{0} N. Akbari and A. R. Ashrafi, Note on the power graph of finite simple groups, \textit{Quasigroups
Related Systems} \textbf{23} (2) (2015) 165--173.

\bibitem{01} D. Bubboloni, M. A. Iranmanesh and S. M. Shaker, Quotient graphs for group power
graphs, to appear in \textit{Rend. Sem. Mat. Univ. Padova}.

\bibitem{02} D. Bubboloni, M. A. Iranmanesh and S. M. Shaker,
On some graphs associated with the finite alternating groups,
\textit{Comm. Algebra} DOI: 10.1080/00927872.2017.1307381.

\bibitem {1} P. J. Cameron and S. Ghosh,  The power
graph of a finite group, \textit{Discrete Math.} \textbf{311}
(2011) 1220--1222.

\bibitem{2} P. J. Cameron,  The power graph of a finite group,
II, \textit{J. Group Theory} \textbf{13} (2010) 779--783.

\bibitem{3}  I. Chakrabarty, S. Ghosh and M. K. Sen, Undirected
power graphs of semigroups, \textit{Semigroup Forum} \textbf{78} (2009)
410--426.

\bibitem{32}  M. Chudnovsky, N. Robertson, P. Seymour and R. Thomas, The strong perfect graph
theorem, \textit{Ann. Math.} \textbf{164} (2006) 51--229.

\bibitem{33} T. Do$\check{\rm s}$li\'c, Splices, links and their degree-weighted Wiener polynomials, \textit{Graph
Theory Notes New York} \textbf{48} (2005) 47--55.


\bibitem{34} A. Hamzeh and A. R. Ashrafi, Automorphism group of supergraphs of the power graph of a finite group,
\textit{European J. Combin.} \textbf{60} (2017) 82–-88.



\bibitem{4} A. V. Kelarev and S. J. Quinn, A combinatorial property and power graphs of semigroups, \textit{Comment. Math. Univ. Carolin.} \textbf{45} (1) (2004) 1--7.

\bibitem{5} A. V. Kelarev and S. J. Quinn, Directed graphs and combinatorial properties of semigroups, \textit{J. Algebra} \textbf{251} (1) (2002) 16--26.

\bibitem{6} A. V. Kelarev and S. J. Quinn,  A combinatorial property and power graphs of groups,
\textit{Contributions to General Algebra} \textbf{12} (Vienna, 1999), 229--235,
Heyn, Klagenfurt, 2000.

\bibitem{7} A. V. Kelarev, S. J. Quinn and R. Smol$\acute{\rm i}$kov$\acute{\rm a}$, Power graphs and semigroups of matrices, \textit{Bull. Austral. Math. Soc.} \textbf{63} (2) (2001) 341--344.

\bibitem{8} Z. Mehranian, A. Gholami and A. R. Ashrafi, A note on the power graph of a finite group,
  \textit{Int. J. Group Theory} \textbf{5} (1) (2016) 1--10.

\bibitem{9} M. Mirzargar, A. R. Ashrafi and M. J. Nadjafi-Arani,  On the power graph of a finite group,
\textit{Filomat} {\bf 26} (2012) 1201--1208.
\bibitem{10} Sh. Payrovi and H. Pasebani, The order graphs of groups, \textit{Alg. Struc. Appl.} \textbf{1} (1) (2014) 1--10.

\bibitem{11} G. R. Pourgholi, H. Yousefi-Azari and A. R. Ashrafi, The undirected power graph of a finite
group, \textit{Bull. Malays. Math. Sci. Soc.} \textbf{38} (4) (2015) 1517--1525.

\bibitem{12} J. S. Rose, \textit{A Course on Group Theory}, Cambridge University Prees, Cambridge, New York-Melbourne, 1978.

\bibitem{12s}   G.  Sabidussi, Graph  Derivatives,  \textit{Math. Z.}  \textbf{76} (1961)  385--401.



\bibitem{14} M. T$\check{\rm a}$rn$\check{\rm a}$uceanu, Some combinatorial aspects of finite Hamiltonian groups, \textit{Bull. Iranian Math. Soc.} \textbf{39} (5) (2013) 841--854.

\bibitem{15} The GAP Team, \emph{GAP -- Groups, Algorithms, and Programming},
Version 4.5.5, 2012, \verb+(http://www.gap-system.org)+.

\bibitem{13} D. B. West, \textit{Introduction to Graph Theory}, Second Edition, Prentice Hall, Inc., Upper Saddle River, NJ, 2001.


\bibitem{16} J. S. Williams, Prime graph components of finite groups, \textit{J. Algebra} \textbf{69} (1981) 487--513.
\end{thebibliography}
\end{document}